\documentclass[11pt,a4paper]{amsart}
\usepackage[utf8]{inputenc}
\usepackage[english]{babel}
\usepackage[T1]{fontenc}
\usepackage{url}
\usepackage{mathrsfs}
\usepackage{tikz}

\title[Hyperbolic knots are not generic]{Hyperbolic knots are not generic}
\author{Yury Belousov and Andrei Malyutin}
\thanks{The research is supported by the Foundation for the Advancement of Theoretical Physics and Mathematics ``BASIS''}

\address{Yury Belousov\\
Faculty of Mathematics, National Research University HSE}
\email{bus99@yandex.ru}

\address{Andrei Malyutin\\
St.\,Petersburg Department of 
Steklov Institute of Mathematics\\
St.~Petersburg State University}
\email{malyutin@pdmi.ras.ru}

\def \crn {\operatorname{cr}}

\newcommand \sm {\setminus}
\newcommand \R  {\mathbb R}

\newcommand \inr{\operatorname{int}}

\newcommand \be     {\begin{equation}}
\newcommand \ee     {\end{equation}}

\newtheorem{thm}{Theorem}

\newtheorem{lemma}{Lemma}

\newtheorem{cor}{Corollary}

\theoremstyle{definition}
\newtheorem{definition}{Definition}

\theoremstyle{remark}
\newtheorem*{remark}{Remark}

\begin{document}
\maketitle

\begin{abstract}
We show that the proportion of hyperbolic knots among all of the prime knots of $n$ or fewer crossings does not converge to $1$ as $n$ approaches infinity. Moreover, we show that if $K$ is a nontrivial knot then the proportion of satellites of~$K$ among all of the prime knots of $n$ or fewer crossings does not converge to $0$ as $n$ approaches infinity.
\end{abstract} 

\maketitle

The present note is intended as an extension and composed as an addendum to~\cite{Mal18} and~\cite{Mal19}, where the question of genericity of hyperbolic knots and links is studied. In particular, we kindly refer the reader to these two papers and references therein for the terminology and introduction to the subject. In this note, we proceed directly with a new lemma (see Lemma~\ref{lem:insoluble} below) we just discovered and explain how to plug this lemma into constructions of~\cite{Mal18} and~\cite{Mal19} in order to prove that the proportion of hyperbolic knots among all of the prime knots of $n$ or fewer crossings does not converge to $1$ as $n$ approaches infinity (this disproves a well-known conjecture; see~\cite[p.\,119]{Ad94}) and, moreover, that for any nontrivial knot~$K$ the proportion of satellites of~$K$ among all of the prime knots of~$n$ or fewer crossings does not converge to $0$ as $n$ approaches infinity.
%, and [[obtain a short new way to new results [upgraded upgrade] see that hyperbolic knots are not generic ]]  

A~brief explanation of the subject matter is as follows. One of the key ingredients of arguments in~\cite{Mal18} and~\cite{Mal19} are specific ways of constructing satellite knots with local satellite structures ($\gamma$-knots and $K$-entanglements) from nonsatellite ones. A~difficult point there is showing that the obtained satellites are prime. 
A~few days ago, we composed a list of various kinds of properties of crossings in knot diagrams for studying Lernaean knots (see~\cite[Definition~1.9]{Mal18}). It turned out that one of these properties, on the one hand, is amenable enough to see that any diagram of a nontrivial knot has a crossing with this property (Lemma~\ref{lem:insoluble}) and, on the other hand, is strong enough to guarantee that each prime knot diagram having a crossing with this property yields at least one prime satellite under our methods of satellite construction.
%our methods of satellite construction when applied to a prime knot diagram having a crossing with this property produce at least one prime satellite
%from a diagram of a prime knot having a crossing with this property, we can construct prime satellites.
This property %of crossings 
is defined as follows.

\begin{definition}[soluble and insoluble crossings]
\label{def:null}
Let $D$ be a knot diagram on $\R^2\subset \R^3$, let $x$ be a crossing of~$D$, and let $\Gamma$ be a knot in~$\R^3$ such that $D$ is its regular projection.
%We denote by~$I_x$ the straight line segment in~$\R^3$ with endpoints at~$\Gamma$ that projects to~$x$.
We say that $x$ is \emph{soluble} if $\R^3$ contains a $3$-ball~$B$ such that 
$(B,B\cap\Gamma)$ is a trivial $1$-string tangle and 
$B$ contains the straight line segment~$I_x$ with endpoints at~$\Gamma$ that projects to~$x$.
Otherwise, we say that $x$ is \emph{insoluble}.
\end{definition}

Presumably, no minimal knot diagram contains a soluble crossing. 
In particular, this is true if the conjecture that the crossing number of knots is additive under connected sum is true.
We prove the following.
%The announced above key lemma is as follows.

\begin{lemma}
\label{lem:insoluble}
Each diagram of any nontrivial knot has an insoluble crossing.
\end{lemma}
\begin{proof}
%"be a simple closed smooth curve" 84
%"be a smooth simple closed curve" 80
%"be a simple smooth closed curve" 39
%"be a smooth closed simple curve" 15
%"be a closed simple smooth curve" 9
%"be a closed smooth simple curve" 5
% "smooth simple closed curve" 126
% "simple closed smooth curve" 119
% "simple smooth closed curve" 97
% "smooth closed simple curve" 71
% "closed simple smooth curve 42
% "closed smooth simple curve" 13
Let $D$ be a diagram of a nontrivial knot~$K$.
We assume that $D$ lies on a plane $\R^2$ contained in~$\R^3$,
and let $\Gamma$ be a simple closed smooth curve in~$\R^3$ %~$\R^3\sm\R^2$ 
representing~$K$ such that $D$ corresponds to~$\Gamma$.
Without loss of generality we can assume that $\Gamma$ does not intersect $\R^2$. 
%Let $\Gamma$ be a simple closed smooth curve in~$\R^3$ representing a nontrivial knot~$K$ and let $D$ be the diagram of~$K$ on $\R^2 \subset \R^3$ determined by~$\Gamma$. Without loss of generality we can assume that $\Gamma$ does not intersect $\R^2$. 
If $x$ is a crossing on $D$, we denote by $e_x$ the nearest to $D$ of the two points of $\Gamma$ projecting to $x$, and let $E$ be the union of $e_x$ for all crossings~$x$ of~$D$. 
An \emph{overarc} on $D$ is the projection of a connected component of $\Gamma \setminus E$. 
We say that two overarcs $\alpha$ and~$\beta$ are \emph{meeting} at a crossing~$x$, if $x$ is an endpoint of~$\alpha$ and $x$ lies on~$\beta$. 
Now, we choose an orientation of~$D$ and consider the corresponding Wirtinger presentation of $\pi_1(\mathbb{R}^3\setminus \Gamma)$ derived from~$D$ (see, e.\,g.,~\cite{BZ06}).
Let $x$ be a crossing on $D$, let $\alpha$ and $\beta$ be overarcs on $D$ meeting at $x$, let $a$ and $b$ be Wirtinger generators corresponding to~$\alpha$ and~$\beta$, respectively, and let $I_x$ be the straight line segment in $\R^3$ with endpoints at~$\Gamma$ that projects to~$x$.
We readily see that any loop representing the element $a b^{-1}$ is freely homotopic to a simple closed curve~$\Delta$ in a small neighborhood of $I_x$. Notice that any loop representing $a b^{-1}$ (and hence $\Delta$) has zero linking number with $\Gamma$. If $x$ is soluble then there exists a ball~$B^3$ such that (i) $(B^3, B^3 \cap \Gamma)$ is a trivial $1$-string tangle, (ii) $I_x \subset B^3$, (iii) $\Delta \subset B^3$.
In this case, $\Delta$ is nullhomotopic (in $B^3\setminus \Gamma$ and hence in $\R^3\setminus \Gamma$). Therefore, we have $ab^{-1}=1$ whenever $a$ and $b$ are Wirtinger generators corresponding to overarcs meeting at a soluble crossing.
Consequently, if all of the crossings in~$D$ are soluble then we have $a=b$ for any pair of Wirtinger generator. In this case, we have $\pi_1(\R^3\setminus \Gamma) = \mathbb{Z}$, which means that $K$ is a trivial knot (see, e.\,g.,~\cite[Proposition 3.17]{BZ06}). 
This contradiction completes the proof. 
\end{proof}

The concept of insoluble crossing is related to the concept of knots and diagrams with weak property~PT introduced in~\cite[Definition 7.1]{Mal18}. We repeat this definition here for the convenience of the reader.

\begin{definition}[weak property PT]
Let $D$ be a knot diagram on the $2$-sphere $S^2=\R^2\cup\{\infty\}$.
We say that $D$ has \emph{weak property PT} (PT stands for ``tangle primeness'') if $D$ is the numerator or denominator closure of a diagram of a locally trivial tangle.
In other words, $D$ has weak property PT if there exists a $2$-disk $d\subset S^2$ such that:
\begin{itemize}
\item the boundary $\partial d$ intersects $D$ transversely in four points;
\item the intersection $d\cap D$ consists of two simple disjoint arcs; 
\item the $2$-string tangle represented by the diagram $\delta \cap D$, where $\delta$ stands for the complementary disk $S^2\setminus \inr(d)$, is locally trivial.
\end{itemize}

We say that a knot has \emph{weak property PT} if it has a minimal diagram with weak property PT.
\end{definition}

Lemma~\ref{lem:insoluble} implies the following corollary.

\begin{cor}
\label{cor:wpPT}
Each diagram of any prime knot has weak property~PT.
In particular, any prime knot has weak property~PT.
\end{cor}

\begin{proof} 
Let $D_P$ be a prime knot diagram on the $2$-sphere $S^2=\R^2\cup\{\infty\}$.
Lemma~\ref{lem:insoluble} says that $D_P$ has an insoluble crossing, say~$x$.
In~$S^2$, let $d$ be a disk containing~$x$ such that the intersection $d \cap D_P$ is homeomorphic to~``$\times$''
and $\partial d$ intersects $D_P$ transversely in four points. 
Let $\delta$ denote the complementary disk $S^2\setminus \inr(d)$, and let $(B,t)$ be the $2$-string tangle represented by the diagram $\delta \cap D_P$. 
Now, we take a subdisk~$d'$ in~$d$ such that the intersection $d' \cap D_P$ consists of 
two subarcs on two distinct legs of $\times=d \cap D_P$ and $\partial d'$ intersects $D_P$ transversely in four points (see Fig.~\ref{fig:d-and-d}).

\begin{figure}[ht]
\centering %\begin{center}
\begin{tikzpicture} 
[knots/.style={white,double=black,very thick,double distance=1.6pt}]
\draw[knots](-0.9,-0.9)--(0.9,0.9);
\draw[knots](-0.9,0.9)--(0.9,-0.9);
\draw[help lines, line width=0.6pt](0,0) circle (5/5);
\draw[help lines, line width=0.6pt] (-45:4/5) arc (-45:45:4/5) arc (45:225:1/5) arc (45:-45:2/5) arc (135:315:1/5);
%\draw[help lines][yscale=4](0.5,0) circle (1/5);
\draw (3/5,0) node {$d'$};
\end{tikzpicture} 
%\end{center}
\caption{Disks~$d$ and $d'$.} 
\label{fig:d-and-d}
\end{figure}
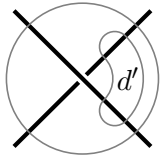

Let $\delta'$ denote the disk $S^2\setminus \inr(d')$.
The diagram $\delta' \cap D_P$ represents a $2$-string tangle homeomorphic to the $2$-string tangle $(B,t)$. 
We show that $(B,t)$ is locally trivial so that $d'$ meets all of the requirements from the definition of weak property~PT.

Suppose on the contrary that $(B,t)$ is locally knotted.
Let $\Gamma$ be a knot in~$\R^3$ such that $D_P$ is its regular projection,
let $I_x$ be the straight line segment with endpoints at~$\Gamma$ that projects to~$x$ (as in Definition~\ref{def:null}), 
let $S^3=\R^3\cup\{\infty\}$ be the one-point compactification of~$\R^3$,
and let $B_0$ be a $3$-ball in~$\R^3$ containing~$I_x$ such that $(B_0,B_0\cap \Gamma)$ is a trivial $2$-tangle with each of its components containing an endpoint of~$I_x$.
We denote the $3$-ball $S^3\sm\inr(B_0)$ by~$B_2$.
Then $(B_2,B_2\cap\Gamma)$ by construction is homeomorphic to $(B,t)$ and hence locally knotted. This means by definition that $B_2$ contains a $3$-ball~$B_1$ such that 
${(B_1,B_1\cap \Gamma)}$ is a nontrivial $1$-tangle.
We denote the $3$-ball $S^3\sm\inr(B_1)$ by~$B$.
Since $\Gamma$ is prime, the Unique Factorization Theorem by~\cite{Sch49} implies that $(B_1,B_1\cap \Gamma)$ is $\Gamma$-knotted and $(B,B\cap\Gamma)$ is a trivial $1$-tangle.
Since $B$ contains $B_0$ and $I_x$, it follows that $x$ is soluble.
This contradiction completes the proof. 
\end{proof}

Corollary~\ref{cor:wpPT} implies Conjectures~10.1 and~10.2 of~\cite{Mal18} (Conjecture~10.1 is precisely Corollary~\ref{cor:wpPT}, and Conjecture~10.2 is a weaker version of Conjecture~10.1).
In particular, Conjecture~10.2 of~\cite{Mal18} is true for $\epsilon=1$. 
As is proved in~\cite{Mal18} (see~\cite[Theorem 10.3]{Mal18}), this implies the following. 

\begin{thm}\label{th:hyperbolic0}
The proportion of hyperbolic knots among all of the prime knots of $n$ or fewer crossings does not converge to $1$ as $n$ approaches infinity. 
Moreover, let $P_n$ {\rm(}\/resp., $H_n$, $S_n$\/{\rm)} denote the number of prime {\rm(}\/resp., hyperbolic, prime satellite\/{\rm)} knots of~$n$ or fewer crossings.
Then
\begin{equation}
\label{eq:S17}
\limsup_{n\to\infty}\frac{S_n}{P_n}~>~\frac{1}{2\cdot 10^{17}},
\end{equation}
and therefore
\begin{equation}
\label{eq:H17}
%\qquad\text{and}\qquad
\liminf_{n\to\infty}\frac{H_n}{P_n}~<~1-\frac{1}{2\cdot 10^{17}}.
\end{equation}
\end{thm}

Furthermore, combining Corollary~\ref{cor:wpPT} with results and arguments of~\cite{Mal18} yields the following more general result. 

\begin{thm}\label{th:satellite}
If $K$ is a nontrivial knot, then the proportion of satellites of~$K$ among all of the prime knots of $n$ or fewer crossings does not converge to~$0$ as $n$ approaches infinity. %More precisely, the following inequalities hold:
More precisely, if $\crn(K)$ is the crossing number of~$K$, $S_n(K)$ is the number of prime satellites of~$K$ with $n$ or fewer crossings, and $\lambda=\limsup_{n\to\infty} \sqrt[n]{P_n}$, then we have
\begin{equation*}
%\label{eq:KSat}
\limsup_{n\to\infty}\frac{S_n(K)}{P_n}~\geq~\frac{1}{1+\lambda^{6(4\crn(K)+1)}}~>~10^{-26\crn(K)}.
\end{equation*}
Furthermore, if $K$ is prime, then
\begin{equation*}
%\label{eq:H17}
\limsup_{n\to\infty}\frac{S_n(K)}{P_n}~\geq~\frac{1}{1+\lambda^{6\crn(K)}}~>~10^{-7\crn(K)}.
\end{equation*}
\end{thm}

In order to prove Theorem~\ref{th:satellite}, it is enough to replace ``prime non-split links'' with ``prime knots'' in the proof of Theorem~1 in~\cite{Mal19} and observe that, by Corollary~\ref{cor:wpPT}, assertion~(iv) of Proposition~1 in~\cite{Mal19} holds for all prime knots.

\begin{remark}
We conjecture that for any nontrivial knot~$K$ the proportion of satellites of~$K$ among all of the prime knots of $n$ or fewer crossings tends to $1$ as $n$ approaches infinity. Certain modifications of techniques developed in~\cite{Mal19} significantly strengthen the estimates of Theorems~\ref{th:hyperbolic0} and~\ref{th:satellite}, however, as far as we know, these techniques are not enough to prove this conjecture.
\end{remark}

\end{document}